\documentclass[12pt]{article}

\usepackage{amsmath,amssymb,amsthm,color}

\newtheorem{Lemma}{Lemma}
\newtheorem{Assumption}{Assumption}
\newtheorem{Proposition}{Proposition}
\newtheorem{Theorem}{Theorem}

\usepackage[T1]{fontenc}
\usepackage[latin1]{inputenc}
\usepackage[round]{natbib}

\newcommand{\ii}{\mathrm{i}}

\newcommand{\vertiii}[1]{{\left\vert\kern-0.25ex\left\vert\kern-0.25ex\left\vert #1 
    \right\vert\kern-0.25ex\right\vert\kern-0.25ex\right\vert}}

\begin{document}

\title{On the CLT for discrete Fourier transforms of functional time series\bigskip}
\author{Cl\'ement Cerovecki$^{1}$ \and Siegfried H\"ormann$^{1}$\thanks{Corresponding author.
Email: shormann@ulb.ac.be}}

\providecommand{\keywords}[1]{\textbf{\textit{Keywords:}} #1}

\date{}
\maketitle
%\vspace*{-1cm}
$^1$ Department of Mathematics, Universit\'e libre de Bruxelles
CP210, Bd.\ du Triomphe, B-1050 Brussels, Belgium.

\maketitle

\vspace*{1cm}
\begin{abstract}
{\bf Abstract.} We consider a strictly stationary functional time series. Our target is to study the weak convergence of the discrete Fourier transforms under sharp conditions. As a side-result we obtain the regular CLT for partial sums under mild assumptions.
\end{abstract}


\section{Introduction} We consider a time series $(X_t\colon t\in\mathbb{Z})$ with realizations in some function space~$H$. Then every observation $X_t$ is a random curve $(X_t(u)\colon u\in \mathcal{U})$ with some continuous domain $\mathcal{U}$. From a technical point of view we require $H$ to be a separable Hilbert space and we call the discrete time process $(X_t\colon t\in\mathbb{Z})$ a {\em functional time series (FTS)}. 

Functional time series analysis is a branch of the emerging field of {\em functional data analysis}. It is not unusual that functional data are sequentially sampled and serially correlated by their very nature. A common situation is that a continuous time process is cut into natural segments, such as days. Then there is not just dependence within the individual curves but also across curves. Nevertheless, growing attention is only given quite recently to this fact. One of the first and most seminal contributors is \citet{bosq:2000}, whose monograph is forming the basic theoretical foundation for FTS. A few recent papers devoted to functional time series are \cite{hyndman:shang:2009}, \cite{horvath:kokoszka:rice:2014}, \cite{horvath:rice:whipple:2015}, \cite{aue:dubart:hormann:2015}.
Some of the latest publications are related to frequency domain topics for FTS. We refer to \cite{panaretos:tavakoli:2013a},  \cite{panaretos:tavakoli:2013b},  \cite{hoermann:kidzinski:hallin:2014}, \cite{hoermann:kidzinski:kokoszka:2014}. In their seminal article on frequency domain methodology for FTS \cite{panaretos:tavakoli:2013a} have studied, among others, the limiting behavior of the discrete Fourier transform of some FTS $(X_t)$: 
$$
S_n(\theta) = \sum_{t=1}^n X_t e^{-\ii t \theta},\quad\theta\in [-\pi,\pi].
$$
This object is of interest to statisticians since it is closely related to the {\em periodogram} which can, for example, be used to detect some underlying periodic behavior of the time series. (See, e.g., \cite{brockwell:davis:1991}.) Applying such a test for periodicity requires knowledge of the distribution of $S_n(\theta)$. But unless $(X_t)$ is a Gaussian process, the exact distribution is infeasible and then we need to rely on asymptotics. Moreover we notice that with $\theta=0$ this framework also contains the regular partial sums process, which is without any doubt a crucial building block in many statistical procedures.

For real valued processes asymptotic normality for $S_n(\theta)$ has been obtained under several dependence conditions. Here we only cite the early paper of \cite{walker:1965} who considered linear processes  and a recent contribution of \cite{wu:2010}, which covers a variety of special cases, including strong mixing sequences. The latter article also contains a more detailed literature survey.
For functional data the afore mentioned paper of \cite{panaretos:tavakoli:2013a} shows that under regularity assumptions $\frac{1}{\sqrt{n}}S_n(\theta)$ converges to a (complex) Gaussian random element whose covariance operator is given by
\begin{equation}\label{Ftheta}
2\pi\mathcal{F}_{\theta}:=\sum_{h\in\mathbb{Z}}C_h e^{-\ii h\theta},
\end{equation}
where $C_h=E\big[X_h\otimes X_0\big]$ is the lag $h$ covariance operator of the stationary functional time series. (It is assumed here and henceforth that $X_t$ are centered by their mean.) The operator $\mathcal{F}_{\theta}$, which can be shown to be self-adjoint and non-negative definite, is called the {\em spectral density operator}. In \cite{panaretos:tavakoli:2013a} it is assumed that $\sum_{h\in\mathbb{Z}}\|C_h\|_{\mathcal{T}}<\infty$ (here $\|\cdot\|_\mathcal{T}$ denotes the Schatten 1-norm) in order to assure convergence of the series defining $\mathcal{F}_{\theta}$. This assumption is convenient, because it implies immediately that $\mathcal{F}_{\theta}$ is a nuclear operator, i.e.\ it has a finite trace. This is an important feature when it comes to verifying tightness of $(S_n(\theta)/\sqrt{n})$.
%Besides the summability condition for the autocovariances stated above, the most important mathematical assumption imposed in \cite{panaretos:tavakoli:2013a} relates to the dependence structure of the process $(X_t)$. 
Regarding the dependence structure, a cumulant type mixing condition for functional data is used. The nice feature of such mixing conditions is that no specific time series model needs to be imposed. Still, this approach requires to compute and bound functional cumulants of all orders, which is generally not an easy task and necessitates moments of all orders. 

One of the main results of this article shows the weak convergence of $S_n(\theta)/\sqrt{n}$ for {\em purely non-deterministic} processes. More precisely, letting  $\mathcal{G}_t=\sigma(X_t,X_{t-1},\ldots)$ and $\mathcal{G}_{-\infty}=\bigcap_{t\geq 0} \mathcal{G}_{-t}$ we impose the following assumption.

\begin{Assumption}\label{ass} The process $(X_t)$ is stationary and ergodic and satisfies
$E[X_0|\mathcal{G}_{-\infty}]=0$ a.s.
\end{Assumption}

We remark that a conditional expectation for random elements in Hilbert spaces  as just stated is well defined if $E\|X_0\|<\infty$ (see e.g.\ \citet[p.29]{bosq:2000}). 
Besides the obligatory existence of second order moments, Assumption~\ref{ass} will be the only condition needed for the CLT presented below in Theorem~\ref{main}. This result is an extension of the advances made in \citet{wu:2010} to infinite dimensional data. Since we will not impose any further condition assuring summability of the $C_h$, a tricky part is the construction and definition of the spectral density operator. Our construction will be an indirect one based on a completeness argument in an appropriate Hilbert space.  

In Theorem~\ref{main2} we will give a result which is slightly less general, but is more useful in applications since it will allow for more explicit constructions of $\mathcal{F}_\theta$. In our Theorem~\ref{main3} we consider the case $\theta=0$ and derive the CLT for regular partial sums. These main results are presented in Section~\ref{se:main}. In Section~\ref{se:examples} we show how the theorems apply in some commonly employed dependence frameworks for functional time series models and compare the required conditions to existing ones in the literature. Proofs are given in Section~\ref{se:proofs}.

\section{Main results}\label{se:main}
We start by introducing further notation and stating the setup precisely. The process $(X_t)$ is defined on some probability space $(\Omega,\mathcal{A},P)$ and takes values in some \emph{separable} \emph{Hilbert space}~$H$. The space $H$ is equipped with inner product $\langle\cdot,\cdot\rangle$ and the resulting norm $\|\cdot\|=\sqrt{\langle \cdot,\cdot\rangle}$. We write $X\in L^p_H(\Omega)$ (short for $X\in L^p_H(\Omega,\mathcal{A},P))$ to indicate that $E\|X\|^p<\infty$.
The space $L^p_H(\Omega)$ is a Banach space and for $p=2$ again a Hilbert space with inner product $E\langle X,Y\rangle$. We assume throughout that $EX_t=0$ (the zero element in $H$) and $X_t\in L^2_H(\Omega)$. Then we denote $C_h=\mathrm{Cov}(X_h,X_0)=E\big[X_h\otimes X_0\big]$. That is, $C_h(u)=EX_h\langle u,X_0\rangle$ for $u\in H$. Expectations or other integrals for elements with values in Banach spaces are understood in the sense of Bochner integrals, see e.g.~\cite{mikusinski}. 

The trace of a self-adjoint and non-negative definite operator $A:H\to H$ is given by $\mathrm{tr}(A)=\|A\|_{\mathcal{T}}=\sum_{j\geq 1}\langle A(v_j),v_j\rangle$ for some orthonormal basis $(v_j)$ of $H$. For $X\in L^2_H(\Omega)$ it holds that $\mathrm{tr}\big(E[X\otimes X]\big)=E\|X\|^2$. For a general operator $A$ we denote by $\|A\|_\mathcal{T}$ the Schatten 1-norm of $A$ and by $\|A\|_\mathcal{S}$ its Hilbert-Schmidt norm.

We use $\mathcal{N}_{H}(\mu,\Sigma)$ to denote a Gaussian element in $H$ with mean $\mu$ and covariance operator $\Sigma$. Then $X\sim \mathcal{N}_{H}(\mu,\Sigma)$ if and only if the projection $\langle X,u\rangle$ is normally distributed with mean $\langle \mu,u\rangle$ and variance $\langle\Sigma(u),u\rangle$ for any $u\in H$. Although our observations are assumed to be real, the very definition of $S_n(\theta)$ necessitates to adopt a complex setting. So we will henceforth assume that the Hilbert space $H=H_0+\ii H_0$ is complex.  Let $EX=\mu=\mu_{\mathrm{Re}}+\ii  \mu_{\mathrm{Im}}$. For $u\in H$ define $\Gamma(u)=E\big[(X-\mu)\langle u,X-\mu\rangle\big]$ and $C(u)=E\big[(X-\mu)\langle X-\mu,u\rangle\big]$. 
 We say that $X$ is complex Gaussian with mean $\mu$, covariance $\Gamma$ and relation operator $C$ if
$$
\begin{pmatrix}\mathrm{Re}(X)\\ \mathrm{Im}(X)\end{pmatrix}
\sim \mathcal{N}_{H_0\times H_0}\left(\begin{pmatrix}\mu_{\mathrm{Re}}\\ \mu_{\mathrm{Im}}\end{pmatrix},\frac{1}{2}\begin{bmatrix} \mathrm{Re}(\Gamma+C)&-\mathrm{Im}(\Gamma-C)\\\mathrm{Im} (\Gamma+C)&\,\,\,\,\,\mathrm{Re}(\Gamma-C)\end{bmatrix}\right).
$$
Henceforth, we will only need the {\em circularly-symmetric case}, i.e.\ when $\mu=0$ and $C=0$. 
Then we write $X\sim \mathcal{CN}_H(0,\Gamma)$. It is straightforward to show that $X\sim \mathcal{CN}_H(0,\Gamma)$, if and only if for any $u\in H$
$$
\begin{pmatrix}\mathrm{Re}(\langle X,u\rangle)\\\mathrm{Im}(\langle X,u\rangle)\end{pmatrix}
\sim \mathcal{N}_{2}\left(\begin{pmatrix}0\\ 0\end{pmatrix},\frac{1}{2}\begin{bmatrix} \langle \Gamma(u),u\rangle&0 \\0&\langle \Gamma(u),u\rangle\end{bmatrix}\right).
$$
%If $(u,v)\in H\times H$ and $V$ a linear operator on $H$, then $V\cdot I_{H\times H}$ is a linear operator on $H\times H$ defined as $V\cdot I_{H\times H}(u,v)=(V(u),V(v))$. 

Finally we recall that a sequence of operators $A_n$ on $H$ is said to converge in the {\em weak operator topology} to $A$ if $\langle A_n(u),v\rangle \to \langle A(u),v\rangle$ for all $u,v\in H$. Short we write $A_n \stackrel{\mathrm{w}}{\longrightarrow} A$. 

%This is our first result.

%\tr{For all $\theta \in [-\pi, \pi]$, we define $S_n(\theta) = \sum_{t=1}^n X_t e^{-\ii t \theta}$. Finally, an operator $T$ on $H$ is said to be \emph{nuclear} if there exists a sequence $(\lambda_j)_{j\geq 1}$ and an Hilbert basis $(v_j)_{j\geq 1}$ such that $T = \sum_{j=1}^{\infty} \lambda_j v_j \otimes v_j$. The space $\mathcal{N}_H$ is a Banach space for the norm (or trace) $\Vert T \Vert_{\mathcal{N}} = \sum_{j=1}^{\infty} \vert \lambda_j \vert$.} 

\begin{Theorem}\label{main} Let $(X_t\colon t\in\mathbb{Z})$ be a sequence in $L_H^2(\Omega)$ which satisfies Assumption~\ref{ass}. Then for almost every $\theta \in [-\pi,\pi]$  there exists a linear operator $\mathcal{F}_\theta$, which is self-adjoint and non-negative definite such that
\begin{equation*}%\label{eq:CLT}
\frac{1}{\sqrt{n}}S_n(\theta) \overset{d}\longrightarrow \mathcal{CN}_{H}\left(0,\pi \mathcal{F}_{\theta}\right).
\end{equation*}
Moreover we have that
\begin{description}
\item[(I)\,\,\,\,] $\frac{1}{n}E\big[S_n(\theta)\otimes S_n(\theta)\big]\stackrel{\mathrm{w}}{\longrightarrow}2\pi\,\mathcal{F}_{\theta}$;
\item[(II)\,\,] $\frac{1}{n}\mathbb{E}\Vert S_n(\theta)\Vert^2_{H}=\frac{1}{n}\mathrm{tr}\big(E\big[S_n(\theta)\otimes S_n(\theta)\big]\big)\to 2\pi\,\mathrm{tr}\big(\mathcal{F}_{\theta}\big)  <\infty$;
\item[(III)] $C_h = \int_{-\pi}^{\pi} \mathcal{F}_{\theta}\, e^{\ii h \theta} d\theta, \qquad \forall h \in \mathbb{Z}$;
\item[(IV)] the components of $\frac{1}{\sqrt{n}}(S_n(\theta),S_n(\theta^\prime))$ are asymptotically independent if $\theta\neq \theta^\prime$.
\end{description}
\end{Theorem}

We call $\mathcal{F}_\theta$ the spectral density operator of $(X_t)$ and remark that it is generally not explicitly defined as in \eqref{Ftheta}. In fact, the series in \eqref{Ftheta} may not be convergent under our mild assumptions. Since
\begin{equation*}%\label{varSn}
2\pi\mathcal{F}_{n;\theta}:=\frac{1}{n}E\big[S_n(\theta)\otimes S_n(\theta)\big]=\sum_{|h|<n}\left(1-\frac{|h|}{n}\right) C_h \exp(-\ii h\theta),
\end{equation*}
relation {\bf (I)} implies solely that the Ces\`aro averages of $(C_h\exp(-\ii h\theta)\colon h\in\mathbb{Z})$ converge (in weak operator topology). Existence of $\mathcal{F}_\theta$ will be obtained via a completeness argument in an appropriate Hilbert space. 

For practical reasons it is useful to know for which frequencies Theorem~\ref{main} holds. For example, $\theta=0$ is an important special case, but we cannot say if this frequency is part of the exceptional null-set or not. We will see that the critical step in the proof of Theorem~\ref{main} is to guarantee existence of the operator $\mathcal{F}_{\theta}$ and to establish the related convergence in \textbf{(I)} and \textbf{(II)}. By the extremely mild assumptions we are imposing, we can only assure this for almost every $\theta$. Requiring Assumption~\ref{ass2} below allows us to establish the same result for some fixed frequency $\theta$. To formulate this assumption we first introduce the projection operator $\mathcal{P}_k:=E[\,\cdot\,|\mathcal{G}_{k}]-E[\,\cdot\,|\mathcal{G}_{k-1}]$, $k\in \mathbb{Z}$.

\begin{Assumption} \label{ass2} 
The process $(X_t)$ is stationary and ergodic and for some selected $\theta\in[-\pi,\pi]$ the following properties hold:
\begin{description}
\item[(A1)] $\sqrt{2\pi}Z_n(\theta):=\sum_{t=0}^n \mathcal{P}_0(X_t)e^{-\ii t\theta}$ is a Cauchy sequence in $L_H^2(\Omega)$;
\item[(A2)] $E\big\|E[S_n(\theta)\vert \mathcal{G}_0]\big\|^2=o(n).$
\end{description}
\end{Assumption}

This is our second main result.

\begin{Theorem}\label{main2} Let $(X_t\colon t\in\mathbb{Z})$ be a sequence in $L_H^2(\Omega)$ which satisfies Assumption~\ref{ass2} for some given $\theta\in[-\pi,\pi]$. Then the conclusions of Theorem~\ref{main} hold for this particular frequency $\theta$.
\end{Theorem}

Assumption~\ref{ass2} looks rather technical, but is relatively easy to verify in a variety of models. (See Section~\ref{se:examples}.)
To get some intuition behind these conditions and the general approach we introduce $Z_m^{(k)}(\theta) :=\frac{1}{\sqrt{2\pi}}\sum_{t=0}^m \mathcal{P}_{k} (X_{t+k}) e^{-\ii t \theta}$ and write
\begin{align}
S_n(\theta)&=\sum_{k=1}^n \mathcal{P}_k(S_n(\theta))+E[S_n(\theta)|\mathcal{G}_0]\nonumber\\
&=\sum_{k=1}^n\sum_{t=k}^n \mathcal{P}_k(X_t)e^{-\ii t\theta}+E[S_n(\theta)|\mathcal{G}_0]\nonumber\\
&=\sqrt{2\pi}\,\sum_{k=1}^n Z^{(k)}_{n-k}(\theta)e^{-\ii k\theta}+E[S_n(\theta)|\mathcal{G}_0].\label{repSn}
\end{align}
The variables $Z_m^{(k)}(\theta)$, $k\geq 1$, are orthogonal and by Assumption {\bf (A1)} they have a limit $Z^{(k)}(\theta)$ in $L_H^2(\Omega)$. Moreover, it is easy to see that $(Z^{(k)}(\theta))_{k\geq 1}$ is a stationary martingale difference sequence. Together with {\bf (A2)} this guarantees that $S_n(\theta)$ is in some sense close to $T_n(\theta):=\sum_{k=1}^n Z^{(k)}(\theta) e^{-\ii k\theta}$. The partial sum $T_n(\theta)$ is more handy when it comes to study the CLT and to compute the covariance operator. In particular, 
\begin{equation}\label{eq:sdo}
\frac{1}{n}\mathrm{Var}\big(T_n(\theta)\big)=\mathrm{Var}\big(Z^{(0)}(\theta)\big)=:\mathcal{F}_{\theta}.
\end{equation}

We conclude this section with a CLT for regular partial sums. It follows as a corollary of Theorem~\ref{main2}.
\begin{Theorem}\label{main3}
Suppose that $(X_t)\in L_H^2(\Omega)$. If Assumption~\ref{ass2} holds with $\theta=0$, then
$$
\sum_{|h|<n}\left(1-\frac{|h|}{n}\right) C_h\stackrel{w}{\to}\mathcal{T},
$$
where $\mathcal{T}$ is some non-negative definite, self-adjoint and trace-class operator. Furthermore it holds that $\mathcal{T}=\lim_{n\to\infty}\mathrm{Var}(\sum_{t=0}^n \mathcal{P}_0(X_t))$ and
$$
\frac{1}{\sqrt{n}}(X_1+\cdots+X_n)\stackrel{d}{\to} \mathcal{N}_H(0,\mathcal{T}).
$$
\end{Theorem}

%Using the second condition of Assumption~\ref{ass2} we obtain
%\begin{align*}\mathbb{E}\Vert S_n(\theta) \Vert^2 &= \mathbb{E}\left\Vert \sum_{k=1}^n \mathcal{P}_k S_n(\theta) + \mathbb{E}[S_n(\theta)\vert \mathcal{G}_0] \right\Vert^2\\
%&=\sum_{k=1}^n \mathbb{E}\Vert Z_0^{(n-k)}e^{-\ii k} \Vert^2  + \mathbb{E}\Vert \mathbb{E}[S_n(\theta)\vert \mathcal{G}_0]\Vert^2 \\
%&= n\left( \mathbb{E}\Vert Z_0(\theta) \Vert^2 + o(1)\right).
%\end{align*}
%This is exactly \textbf{(II)} for a given $\theta$, indeed $\mathbb{E}\Vert Z_0(\theta) \Vert^2 = 2\pi \mathrm{tr}(\mathcal{F}_{\theta})$. For \textbf{(I)} we proceed similarly with the projections $\langle S_n(\theta),v \rangle$.\medskip

\section{Examples}\label{se:examples}

The purpose of this section is to show that our results apply to some commonly used dependence models for functional data. The typical framework we have in mind comprises processes $(X_t)$ which can be represented as Bernoulli shifts, i.e.\
\begin{equation}\label{bernoulli}
X_t=f(\varepsilon_t,\varepsilon_{t-1},\ldots)
\end{equation}
where $(\varepsilon_t)$ is a stationary and ergodic sequence of elements in some normed vector space $S$ and $f:S^\mathbb{N}\to H$ is measurable. Then $(X_t)$ is again stationary and ergodic. We remark that in this case we can use in our theorems the filtration $(\mathcal{G}_k\colon k\in\mathbb{Z})$ with $\mathcal{G}_k=\sigma(\varepsilon_k,\varepsilon_{k-1},\ldots)$. When $(\varepsilon_t)$ are i.i.d., then by Kolmogorov's 0-1 law Assumption~\ref{ass} applies to all such processes. Therefore, only requiring the necessary condition $X_t\in L^2_H(\Omega)$ already implies Theorem~\ref{main}. But $(\varepsilon_t)$ need not necessarily be independent. For example if $(\varepsilon_t)$ are \emph{strongly mixing} then the tail sigma algebra $\mathcal{G}_{-\infty}$ is again trivial (see \citet[p.10]{bradley:2005}) and hence Assumption~\ref{ass} still holds.

Representation \eqref{bernoulli} is very common to many time series models. In particular it applies to the two dependence frameworks we are going to discuss below, namely \emph{linear processes} (possibly with dependent noise) and \emph{$L^2-m$--approximable processes}. These two concepts cover most of the functional time series models studied in the literature. Let us beforehand give a simple condition which can replace Assumption~\ref{ass2}. Here and in the sequel $\nu_p(X)=\big(E\|X\|^p\big)^{1/p}$, $p\geq 1$.
\begin{Lemma}\label{le:suff}
Assumption~\ref{ass} and condition {\bf (A3)}: $\sum_{t=0}^{\infty}\nu_2(\mathcal{P}_0(X_t))<\infty$ together imply Assumption~\ref{ass2} for all $\theta\in [-\pi,\pi]$.
\end{Lemma}
\begin{proof}
It is easy to see that \textbf{(A3)} implies \textbf{(A1)} for all $\theta \in [-\pi,\pi]$. 

Consider the decaying sequence of $\sigma$-algebras $(\mathcal{G}_{-k}\colon k\geq 0)$. For any integrable random variable $X\in H$ the process $E[X|\mathcal{G}_{-k}]$ is a reverse martingale with values in $H$ (see e.g.\ \cite{chatterji:1964}) and consequently the increments $\mathcal{P}_{-k}(X)=E[X|\mathcal{G}_{-k}]-E[X|\mathcal{G}_{-k-1}]$ are orthogonal elements in $L^2_H(\Omega)$. Furthermore, by Assumption~\ref{ass} it then follows that $E[X|\mathcal{G}_{-k}]$ converges a.s.\ and in $L^2_H(\Omega)$ to $E[X|\mathcal{G}_{-\infty}]=0$. Hence, with $X=S_n(\theta)$ we get
\begin{align*}
E\|E[S_n(\theta)|\mathcal{G}_0]\|^2&=\lim_{n\to\infty}E\|E[S_n(\theta)|\mathcal{G}_0]-E[E[S_n(\theta)|\mathcal{G}_0]|\mathcal{G}_{-n}]\|^2\\
& = \sum_{j\geq 0} E\|\mathcal{P}_{-j}(E[S_n(\theta)|\mathcal{G}_0])\|^2=\sum_{j\geq 0} E\|\mathcal{P}_{-j}(S_n(\theta))\|^2.
\end{align*}
Therefore we may conclude that
\begin{align*}
E\Vert E[S_n(\theta)\vert \mathcal{G}_0]\Vert^2
&\leq  \sum_{j=0}^{\infty}E \sum_{s,t=1}^n \vert \langle \mathcal{P}_{-j}(X_t) , \mathcal{P}_{-j}(X_s)  \rangle\vert \\
& = \sum_{j=0}^{\infty}\sum_{s,t=1}^n E\vert \langle \mathcal{P}_{-j}(X_t) , \mathcal{P}_{-j}(X_s)  \rangle\vert \\
& \leq \sum_{s=1}^n\sum_{j=0}^{\infty}\left(\sum_{t=0}^\infty \nu_2( \mathcal{P}_0 (X_{t+j}))\right) \nu_2(\mathcal{P}_0 (X_{s+j}) ) = o(n).
\end{align*}
\end{proof}

%\begin{align}\label{A1}\tag{A1} &\mathbb{E}[X_0|\mathcal{G}_{-\infty}]=0,\quad \mathbb{P}\text{-a.s.} \\
%\label{A2}\tag{A2}&\exists \xi_0(\theta)\in\mathbb{L}_H^2(\Omega), \quad \sum_{t=0}^n \mathcal{P}_0(X_t)e^{-\ii t\theta}  \underset{n\to \infty}\longrightarrow  \xi_0(\theta) \quad \text{in} \quad \mathbb{L}_H^2(\Omega). \\
%\label{A3}\tag{A3} &\mathbb{E}\left[\frac{S_n(\theta)}{\sqrt{n}}\bigg\vert \mathcal{G}_0\right]  \underset{n\to \infty}\longrightarrow  0 \quad \text{in} \quad \mathbb{L}_H^2(\Omega).\\
%\label{A4}\tag{A4} &\exists \mathcal{F}_{\theta}\in \mathcal{N}_H,\quad  \mathrm{Tr}(\mathcal{F}_{\theta ,n}) \underset{n\to\infty}\longrightarrow \mathrm{Tr}(\mathcal{F}_{\theta}). \\
%\label{A5}\tag{A5}&\exists \mathcal{F}_{\theta}\in \mathcal{N}_H,\quad  \mathcal{F}_{\theta ,n} \overset{WOT}{\underset{n\to\infty}\longrightarrow} \mathcal{F}_{\theta}. \\
%\label{A6}\tag{A6} & \sum_{h\in\mathbb{Z}}\mathrm{tr}(C_h)<\infty.
%\end{align}

\noindent
\emph{Example~1: Linear processes.}\bigskip

\noindent 
Consider a linear process $X_t=\sum_{k\geq 0} \Psi_k(\varepsilon_{t-k})$ where $(\varepsilon_t)_{t\in \mathbb{Z}}$ are i.i.d.\ and zero mean in some Hilbert space $H^\prime$ and $\Psi_k:H^\prime\to H$ are bounded linear operators. We denote by $\|\Psi\|_\mathcal{L}$ the operator norm.
\begin{Theorem}\label{main4} If $X_t\in L^2_H(\Omega)$ then Theorem~\ref{main} holds. If in addition
$
\kappa:=\sum_{k\geq 0}\|\Psi_k\|_\mathcal{L}<\infty
$
then Theorem~\ref{main2} applies for all $\theta\in[-\pi,\pi]$. Moreover
$$
\mathcal{F}_\theta=\frac{1}{2\pi}\Psi(\theta)V\Psi(\theta)^*,
$$
where $\Psi(\theta)=\sum_{k\geq 0}\Psi_k e^{-\ii k\theta}$ and $\Psi(\theta)^*$ is its adjoint operator and $V=\mathrm{Var}(\varepsilon_0):=E\big[\varepsilon_0\otimes\varepsilon_0\big]$.
\end{Theorem}
It is easy to see that $\varepsilon_0\in L_{H^\prime}^2(\Omega)$ implies that $X_t\in L^2_H(\Omega)$. Consequently, our Theorem~\ref{main4} improves the corresponding result in \cite{panaretos:tavakoli:2013a}, where it is required that $\varepsilon_0\in L_{H^\prime}^k(\Omega)$ for all $k\geq 1$.%, and that $\sum_{\ell\geq 0}(1+\ell)\|\Psi_\ell\|_\mathcal{S}<\infty$. (Here $\Psi_\ell$ are supposed to be Hilbert-Schmidt operators).

When $\theta=0$ we recover the ordinary CLT for the partial sums of $(X_t)$ as e.g.\ proven in \cite{merlevede:peligrad:utev:1997}. While for real-valued linear processes the CLT only requires square summability of coefficients, the latter authors prove that in infinite dimensional Hilbert spaces assuming absolute summability is essentially sharp.
\begin{proof}
%We assume without loss of generality that $\mathrm{tr}(V)=E\|\varepsilon_0\|^2=1$. Noting that $\mathcal{P}_0(X_t)=\Psi_t(\varepsilon_0)$ and setting $\Psi^{m,n}_\theta=\sum_{s=m+1}^n\Psi_s e^{-\ii s\theta}$ we infer
%\begin{align*}
%E\langle \Psi^{m,n}_\theta(\varepsilon_0),\Psi^{m,n}_\theta(\varepsilon_0)\rangle
%&=\mathrm{tr}\left(\big(\Psi^{m,n}_\theta\big)V\big(\Psi^{m,n}_\theta\big)^*\right)\\
%&\leq\left\|\Psi^{m,n}_\theta\right\|_\mathcal{L}^2\leq \bigg(\sum_{t>m}\|\Psi_t\|_\mathcal{L}\bigg)^2\to 0\quad (m\to\infty),
%\end{align*}
%and {\bf(A1)} follows. Moreover 
%\begin{align*} \mathbb{E}\Vert \mathbb{E}[S_n(\theta)\vert\mathcal{G}_0] \Vert^2 &= \sum_{s,t= 1}^n \mathbb{E}\langle\mathbb{E}[X_s\vert \mathcal{G}_0] ,\mathbb{E}[X_t\vert \mathcal{G}_0] \rangle e^{-\ii (t-s)\theta} \\
%& = \sum_{s,t= 1}^n \mathbb{E}\left\langle \sum_{k=s}^{\infty} \Psi_k (\varepsilon_{s-k}),\sum_{j=t}^{\infty} \Psi_j (\varepsilon_{t-j}) \right\rangle e^{-\ii (t-s)\theta}\\
%&=  \sum_{s,t= 1}^n \sum_{j\geq 0}\mathbb{E}\left\langle  \Psi_{s+j} (\varepsilon_{-j}), \Psi_{t+j} (\varepsilon_{-j}) \right\rangle e^{-\ii (t-s)\theta}\\
%&\leq\sum_{s,t= 1}^n\sum_{j\geq 0}\|\Psi_{j+s}\|_\mathcal{L}\|\Psi_{j+t}\|_\mathcal{L}\\
%&\leq\kappa\sum_{s=1}^n\left(\sum_{j\geq 0}\|\Psi_{j+s}\|_\mathcal{L}\right)=o(n).
%\end{align*}

Noting that $\mathcal{P}_0(X_t)=\Psi_t(\varepsilon_0)$ condition {\bf (A3)} follows immediately from Lemma~\ref{le:suff}. Also, $\kappa<\infty$ implies $C_h=\sum_{k\geq 0}\Psi_{k+h}V\Psi_k^*$ and this in turn yields that $\sum_{h\in\mathbb{Z}}\|C_h\|_\mathcal{L}<\infty$. Therefore $\mathcal{F}_\theta$ has representation~\eqref{Ftheta}.
\end{proof}\bigskip

\noindent
\emph{Example~2: $L^2-m$--approximable processes.}\bigskip

\noindent 
\cite{hormann:kokoszka:2010} have used the concept of
 $L^p-m$--\emph{approximabil\-ity} for analyzing dependent functional data. Then a process $(X_t)$ is said to be $L^p-m$--approximable if $X_t$ has representation \eqref{bernoulli} with i.i.d.\ innovations and
$$
\sum_{m=1}^{\infty} \nu_p(X_0-X_0^{(m)})<\infty,
$$
where  $X_0^{(m)}=f(\varepsilon_{0},\dots,\varepsilon_{0-m+1},\tilde{\varepsilon}_{-m},\tilde{\varepsilon}_{-m-1}\dots)$
for some independent copy $(\tilde{\varepsilon}_t)$ of $(\varepsilon_t)$.  In \cite{hormann:kokoszka:2010} it is shown that this concept applies to many stationary and non-stationary functional time series models, including, for example, functional ARCH. The concept is somewhat related to {\em near epoch dependence (NED)} often employed in the econometrics literature. See, e.g., \cite{poetscher:prucha:1997}. If this condition holds with $p>2$, then by a recent result of \cite{berkes:horvath:rice:2013} a weak invariance principle for the partial sums process holds. This result has been sharpened by \cite{jirak:2013} who proved the same invariance principle under $p=2$ and also under a milder coupling condition.
\begin{Theorem}\label{th:5} Suppose that $(X_t)$ is $L^2-m$--approximable. Then Theorem~\ref{main2} applies for all $\theta\in[-\pi,\pi]$. Moreover, $\sum_{h\in\mathbb{Z}}\|C_h\|_\mathcal{S}<\infty$ and therefore
$$
\mathcal{F}_\theta=\frac{1}{2\pi}\sum_{h\in\mathbb{Z}}C_he^{-\ii\theta h}.
$$
\end{Theorem}

\begin{proof}
We first note that under $L^2-m$-approximability $X_s^{(s)}$ is independent of $\mathcal{G}_0$. Hence we get
\begin{align*}{E}\Vert \mathcal{P}_0(X_s) \Vert^2 &= {E}\Vert \mathcal{P}_0 (X_s-X_s^{(s)}) \Vert^2\\
 &\leq 2  {E} \left( \Vert {E}[ X_s-X_s^{(s)}\vert \mathcal{G}_0]\Vert^2 + \Vert {E}[ X_s-X_s^{(s)}\vert \mathcal{G}_{-1}]\Vert^2 \right)\\
 &\leq 4 {E}\Vert X_0-X_0^{(s)} \Vert^2=4\nu_2^2\big(X_0-X_0^{(s)}\big).
\end{align*}
Now apply Lemma~\ref{le:suff}.
%It then follows by the Cauchy-Schwarz inequality that for $m<n$
%\begin{align*}
%{E}\left\Vert Z_n(\theta)-Z_m(\theta) \right\Vert^2 &\leq \sum_{m<s,t\leq n} \vert {E}\langle \mathcal{P}_0 (X_s) ,\mathcal{P}_0 (X_t) \rangle \vert \\
%&\leq 4\left(\sum_{s\geq m}  \nu_2\big(X_0-X_0^{(s)}\big)\right)^2\to 0\quad (m\to\infty).
%\end{align*}
%This proves {\bf (A1)}. To show {\bf (A2)} we observe that similar arguments yield
%\begin{align*}
%{E}\left\Vert {E}[S_n(\theta)\vert \mathcal{G}_0] \right\Vert^2& \leq \sum_{s,t= 1}^n \left\vert {E}\langle{E}[X_s\vert \mathcal{G}_0] ,{E}[X_t\vert \mathcal{G}_0] \rangle \right\vert\\
% &= \sum_{s,t= 1}^n  \vert {E}\langle{E}[X_s-X_s^{(s)}\vert \mathcal{G}_0] ,{E}[X_t-X_t^{(t)}\vert \mathcal{G}_0] \rangle \vert \\
%&\leq \left(\sum_{s=1}^n \nu_2\big(X_0-X_0^{(s)}\big) \right)^2=O(1).
%\end{align*}

The absolute summability condition on the $C_h$ has been derived in \citet{hoermann:kidzinski:hallin:2014}.
\end{proof}\bigskip

\noindent
\emph{Example 3: Linear processes with dependent errors}\bigskip

\noindent
Consider once again a linear process $X_t=\sum_{k\geq 0} \Psi_k(\varepsilon_{t-k})$, where now $(\varepsilon_t)_{t\in \mathbb{Z}}$ is a stationary and ergodic zero mean sequence.

\begin{Theorem}\label{th:6} Suppose that $(X_t)$ is a linear process satisfying the summability condition $\sum_{k\geq 0}\|\Psi_k\|_\mathcal{L}<\infty$. Assume moreover that $(\varepsilon_t)_{t\in \mathbb{Z}}$ satisfies {\bf (A3)}. Then the conclusion of Theorem~\ref{th:5} holds.
\end{Theorem}

For the regular partial sums process, this result compares to \cite{rackauskas:suquet:2010} who have studied partial sums of linear processes in Banach spaces. They show that the CLT for the innovations transfers to the linear process under summability of $(\|\Psi_k\|_\mathcal{L}\colon k\geq 0)$.
\begin{proof}
We apply Lemma~\ref{le:suff}. To this end note that
\begin{align*}
\sum_{t\geq 0}\nu_2\big(\mathcal{P}_0(X_t)\big)&\leq \sum_{t\geq 0}\sum_{k\geq 0}\nu_2\big(\mathcal{P}_0(\Psi_k(\varepsilon_{t-k})\big)\\
&\leq \sum_{k\geq 0}\|\Psi_k\|_\mathcal{L} \ \sum_{t\geq 0} \nu_2\big(\mathcal{P}_0(\varepsilon_{t-k})\big)<\infty.
\end{align*}
\end{proof}

\section{Proofs}\label{se:proofs}

\subsection{Preliminary lemmas}\label{1}
Let $G=L^2_H(\Omega)$ and consider the Hilbert space $L^2_G([-\pi,\pi],\mathcal{B},\lambda)$, with $\mathcal{B}$ and $\lambda$ being the Lebesgue measure and the Borel $\sigma$-field on $[-\pi,\pi]$, respectively. For simplicity we write $L^2_G([-\pi,\pi])$. This space is equipped with inner product
$(V,W)=\int_{-\pi}^\pi E\langle V(\theta),W(\theta)\rangle d\theta$ and norm $\vertiii{V}=\sqrt{(V,V)}$.

\begin{Lemma}\label{le1}
Define $Z_n=Z_n(\theta):=\frac{1}{\sqrt{2\pi}}\sum_{t=0}^n \mathcal{P}_0(X_t)e^{-\ii t\theta}$. Then $(Z_n)$ is a Cauchy sequence in $L^2_G([-\pi,\pi])$, if and only if $\sum_{t\geq 0} E\|\mathcal{P}_{-t}(X_0)\|^2<\infty$. Moreover, under Assumption~\ref{ass} the latter summability condition holds.
\end{Lemma}
We remark that this lemma provides a slightly weaker version of Assumption~\ref{ass2}, part 1. 
\begin{proof}
Using stationarity and the orthogonality of the functions $\theta\mapsto e^{-\ii t\theta}$, $\theta\in [-\pi,\pi]$, $(t\in\mathbb{Z})$ we obtain for $m<n$
\begin{align*}
\vertiii{Z_n-Z_m}^2&=\frac{1}{2\pi}\int_{-\pi}^\pi E\left\|\sum_{t=m+1}^n \mathcal{P}_0 (X_t) e^{-\ii t\theta}\right\|^2_Hd\theta\\
&=\sum_{t=m+1}^nE\|\mathcal{P}_0(X_t)\|^2=\sum_{t=m+1}^nE\|\mathcal{P}_{-t}(X_0)\|^2.
\end{align*}
%Consider the decaying sequence of $\sigma$-algebras $(\mathcal{G}_{-k}\colon k\geq 0)$. The process $E[X_0|\mathcal{G}_{-k}]$ is a reverse martingale with values in $H$ (see e.g.\ Chatterji~\cite{chatterji:1964}) and consequently the increments $\mathcal{P}_{-k}(X_0)=E[X_0|\mathcal{G}_{-k}]-E[X_0|\mathcal{G}_{-k-1}]$ are orthogonal elements in $L^2_H(\Omega)$. Furthermore, $E[X_0|\mathcal{G}_{-k}]$ converges a.s.\ and in $L^2_H(\Omega)$ to $E[X_0|\mathcal{G}_{-\infty}]=0$. Hence
By the arguments in the proof of Lemma~\ref{le:suff} we have
$$
\sum_{t\geq 0} E\|\mathcal{P}_{-t}(X_0)\|^2 =\lim_{n\to\infty}E\|X_0-E[X_0|\mathcal{G}_{-n}]\|^2 =
E\|X_0\|^2 \; < \infty.
$$
%To see this in detail we remark that $\langle V_k^\prime,V_\ell^\prime\rangle_G=E\sum_{j\geq 0}\langle e_j,V_k^\prime \rangle \langle e_j,V_\ell^\prime\rangle$. Using dominated convergence (the sum is bounded by $\|V_k^\prime\|\|V_\ell^\prime\|$) we can interchange expectation and summation. It is easily seen (use projection theorem) that $(\langle u,V_k^\prime\rangle)$ are orthogonal. It follows that $E\langle u, V_k^\prime\rangle \langle u, V_\ell^\prime\rangle=0$ for all $u\in H$ if $k\neq \ell$.
\end{proof}

It follows under Assumption~\ref{ass} that there exists an element $Z\in L^2_G([-\pi,\pi])$ with $\vertiii{Z_n-Z}\to 0$.
This in turn has some important implications.\medskip

\noindent
(i)\ Since 
$$
\vertiii{Z}^2=\int_{-\pi}^\pi E\|Z(\theta)\|^2 d\theta<\infty,
$$
we conclude that $E\|Z(\theta)\|^2<\infty$ for all $\theta \in M_0=[-\pi,\pi]\backslash N_0$ where $\lambda(N_0)=0$. Hence, for all $\theta \in M_0$ the covariance operator $\mathcal{F}_{\theta}:= E \big[Z(\theta)\otimes Z(\theta)\big]$ is well defined, self-adjoint and non-negative definite. The denotation $\mathcal{F}_\theta$ is intentional. As we will see later {\em it is defining the spectral density operator} (compare to \eqref{eq:sdo}). Since $\mathrm{tr}\big(\mathcal{F}_{\theta}\big)=E\|Z(\theta)\|^2$, this operator is trace class. For $\theta\in N_0$ we set $\mathcal{F}_{\theta}=0$.\medskip

\noindent
(ii)\ There exists a sequence $(n_k)$ such that $E\|Z_{n_k}(\theta)-Z(\theta)\|^2\to 0$ for all $\theta \in M_1:= [-\pi,\pi]\backslash N_1$, where $\lambda(N_1)=0$.\medskip

\noindent
(iii) By construction the mapping $\theta\mapsto Z(\theta)\in G$ is measurable, and the mapping $Z(\theta):G\to \mathcal{S}$ (the set of Hilbert-Schmidt operators on $H$), $Z(\theta)\mapsto E\big[Z(\theta)\otimes Z(\theta)\big]$, is continuous. Hence, $\theta \to\mathcal{F}_{\theta}$ is measurable as a mapping from $[-\pi,\pi]$ to the space $\mathcal{S}$, which is known to be a separable Hilbert space. Consequently the integral in {\bf (III)} of Theorem~\ref{main} is well defined.\medskip

The next lemma will be used in the proof of tightness and implies part {\bf (II)} of Theorem~\ref{main}.

\begin{Lemma}\label{le1b} Under
Assumption~\ref{ass} we have for all $\theta \in M_2=[-\pi,\pi]\backslash N_2$ with $\lambda(N_2)=0$ that
$$\mathrm{tr}\big(\mathcal{F}_{n;\theta}\big)=\frac{1}{2\pi}\sum_{|h|<n}\left(1-\frac{|h|}{n}\right)E\langle X_h,X_0\rangle e^{-\ii h\theta}\to \mathrm{tr}\big(\mathcal{F}_{\theta}\big)<\infty.$$ 
\end{Lemma}
\begin{proof}
%We recall that $E\|Z(\theta)\|^2=\|\mathcal{F}_{\theta}\|_\mathcal{T}$. 
%From (i) we know that $$\int_{-\pi}^\pi E\|Z(\theta)\|^2 d\theta<\infty.$$ 
Set $c_h:=\int_{-\pi}^\pi E\|Z(\theta)\|^2 e^{\ii h\theta}d\theta$. Using (i) we infer from the Fej\'er-Lebesgue theorem that
\begin{equation*}%\label{lf}
\sum_{|h|<n}\left(1-\frac{|h|}{n}\right)c_h e^{-\ii h\theta}\to E\|Z(\theta)\|^2=\mathrm{tr}\big(\mathcal{F}_\theta\big) <\infty\quad \text{for almost all $\theta$.}
\end{equation*}
We define $M_2$ as the set of convergence points.
We show now that $c_h=E\langle X_h,X_0\rangle$. Using Lemma~\ref{le1} and continuity of $\vertiii{\,\cdot\,}$, it can be readily shown that
$$
c_h=\lim_{n\to\infty}\int_{-\pi}^\pi E\|Z_n(\theta)\|^2 e^{\ii h\theta}d\theta.
$$
%\tr{Indeed, we can split $e^{\ii h \theta}$ into $\cos^{+}(h\theta)-\cos^{-}(h\theta)+\ii \sin^{+}(h\theta)-\ii \sin^{-}(h\theta)$. It can then be shown that $Z_n$ converges also to $Z$ in $L^2_G([-\pi,\pi],\mu)$ where $\mu$ is a measure with density being $\cos^{+}(h\theta)$, $\cos^{-}(h\theta)$, $\sin^{+}(h\theta)$ or $\sin^{-}(h\theta)$. The result thus follows from the continuity of the norm of $L^2_G([-\pi,\pi],\mu)$.}
Without loss of generality assume $h\geq 0$. Using stationarity we deduce
\begin{align*}
&\int_{-\pi}^\pi E\|Z_n(\theta)\|^2 e^{\ii h\theta}d\theta=\sum_{t=0}^n\sum_{s=0}^nE\langle \mathcal{P}_0(X_t),\mathcal{P}_0(X_s)\rangle\frac{1}{2\pi}\int_{-\pi}^\pi e^{-\ii(t-s-h)\theta}d\theta\\
&\qquad=\sum_{t=h}^n E\langle \mathcal{P}_0(X_t),\mathcal{P}_0(X_{t-h})\rangle=\sum_{t=h}^n E\langle \mathcal{P}_{-t}(X_0),\mathcal{P}_{-t}(X_{-h})\rangle.
\end{align*}
Since the terms $\mathcal{P}_{-t}(X_0)$ and $\mathcal{P}_{-s}(X_{-h})$ are orthogonal in $L_H^2(\Omega)$ for $s\neq t$, it follows that 
\begin{align*}
&\int_{-\pi}^\pi E\|Z_n(\theta)\|^2 e^{\ii h\theta}d\theta=E\left\langle\sum_{t=h}^n \mathcal{P}_{-t}(X_0),\sum_{s=h}^n  \mathcal{P}_{-s}(X_{-h})\right\rangle.
\end{align*}
We recall that $(E[X_0|\mathcal{G}_{-t}]\colon t\geq 0)$ is a reverse martingale, it follows by Assumption~\ref{ass} that 
\begin{equation}\label{revmarconv}\sum_{t=h}^n \mathcal{P}_{-t}(X_0)=E[X_0|\mathcal{G}_{-h}]-E[X_0|\mathcal{G}_{-(n+1)}]\stackrel{L_H^2(\Omega)}{\longrightarrow}E[X_0|\mathcal{G}_{-h}].
\end{equation}
Similarly, $\sum_{s=h}^n  \mathcal{P}_{-s}(X_{-h}) \stackrel{L_H^2(\Omega)}{\longrightarrow} E[X_{-h}|\mathcal{G}_{-h}]=X_{-h}.
$
And hence, by continuity of the inner product, $c_h=E\big\langle E[X_0|\mathcal{G}_{-h}], X_{-h}\big\rangle=E\langle X_h,X_0\rangle$.
\end{proof} 

%The following lemma (see \cite[p.29]{bosq:2000}) which will be used at several places is stated for convenience of the reader.
%\begin{Lemma}\label{le2}
%Let $L:B_1\to B_2$ be a linear operator between two \emph{separable Banach spaces}, $\mathcal{G}\subset \mathcal{A}$ some $\sigma$-algebra and $X\in L_{B_1}^1(\Omega)$. Then $L(E[X|\mathcal{G}])=E[L(X)|\mathcal{G}]$.
%\end{Lemma}

\begin{Lemma}\label{le5} The operators $\mathcal{F}_{\theta}$ define the \emph{spectral density operators} of $(X_t)$ at frequency $\theta$. This is
\[ C_h= \int_{-\pi}^{\pi} \mathcal{F}_{\theta} \,e^{\ii h \theta} d\theta , \qquad \forall h \in \mathbb{Z}.\]
\end{Lemma}
\begin{proof} We have seen in (iii) that the mapping $\theta \mapsto \mathcal{F}_{\theta}$ is measurable. The integrand is valued in the separable Hilbert space $\mathcal{S}$. %This special case of Bochner integral is thus very similar with the expectation of random element in a separable Banach space we have used so far and whose construction can be found e.g. in \cite{bosq:2000}.
Since $\int_{-\pi}^{\pi}\|\mathcal{F}_\theta\|_\mathcal{S}d\theta\leq \int_{-\pi}^{\pi}\mathrm{tr}\big( \mathcal{F}_{\theta}\big)d \theta < \infty$ we know that $\mathcal{F}_\theta$ is strongly integrable and hence we can define (in the sense of a Bochner integral) $I=\int_{-\pi}^{\pi} \mathcal{F}_{\theta}\, e^{\ii h \theta} d \theta$. Let $u,v \in H$. Since Bochner integrals are interchangeable with bounded linear operators we obtain
\begin{align*} \langle I (v) , u \rangle &= \int_{-\pi}^{\pi} \langle \mathcal{F}_{\theta}(v),u\rangle e^{\ii h \theta} d\theta  = \int_{-\pi}^{\pi} {E}\langle Z(\theta),u\rangle \overline{\langle Z(\theta),v\rangle}_H e^{\ii h \theta} d\theta \\
&= \lim_{n\rightarrow \infty}\int_{-\pi}^{\pi} {E}\langle Z_{n}(\theta),u\rangle \: \overline{\langle Z_{n}(\theta),v\rangle}_H e^{\ii h \theta} d\theta.
\end{align*}
The last equality can be deduced from $\vertiii{Z_n-Z}\to 0$. Assume now with loss of generality that $h\geq 0$. Similar arguments as in Lemma \ref{le1b} lead to
\begin{align*}
\langle I (v) , u \rangle %&=\lim_{n\rightarrow \infty}\frac{1}{2\pi}\int_{-\pi}^{\pi} \sum_{t,s=0}^n {E}\langle u,\mathcal{P}_0 X_t\rangle \overline{\langle v,\mathcal{P}_0 X_s\rangle} e^{-\ii (t-s-h) \theta} d\theta \\
%&= \lim_{n\rightarrow \infty} \sum_{t,s=0}^n {E}\langle u,\mathcal{P}_0 X_t\rangle \overline{\langle v,\mathcal{P}_0 X_s\rangle} \frac{1}{2\pi}\int_{-\pi}^{\pi} e^{-\ii (t-s-h) \theta} d\theta \\
%&= \lim_{n\rightarrow \infty} \sum_{t=h}^n {E} \langle u,\mathcal{P}_{-t} X_{-h}\rangle \overline{\langle v,\mathcal{P}_{-t} X_{0}\rangle} \\
&= \lim_{n\rightarrow \infty} {E} \bigg[\bigg\langle  \sum_{t=0}^{n-h} \mathcal{P}_{-t}(X_h),u\bigg\rangle \: \bigg\langle v,\sum_{s=0}^{n-h}\mathcal{P}_{-s}(X_0) \bigg\rangle\bigg].
\end{align*}
From \eqref{revmarconv} it follows that
\begin{align*}
\langle I (v) , u \rangle = {E}\big[ \langle  E[X_{h}|\mathcal{G}_0],u \rangle \: \langle v, X_0\rangle\big] &= \langle C_h(v)  ,u \rangle.
\end{align*}
Since $u$ and $v$ are arbitrary in $H$, we can infer that $I = C_h$. 
\end{proof}

We show next that the projections $\langle S_n(\theta),u\rangle/\sqrt{n}$, $u\in H$, converge weakly to $\langle S_0(\theta),u\rangle$, where $S_0(\theta)\sim \mathcal{CN}_{H}(0,\pi\mathcal{F}_\theta)$ is the limiting complex Gaussian element and where $\mathcal{F}_\theta$ is defined in (i). The first step towards this result is given by the following proposition.
\begin{Proposition}\label{prop1} 
Under Assumption~\ref{ass} there exists for all $u\in H$ a set $\widetilde N\subset [-\pi,\pi]$ of Lebesgue measure 0, such that on $\widetilde M=[-\pi,\pi]\backslash \widetilde N$ the following holds:
\begin{description}
\item[(a)] $\lim_{n\to\infty} \mathrm{Var}(\langle S_n(\theta),u\rangle)/n=\lim_{n\to\infty}2\pi \langle \mathcal{F}_{n;\theta}(u),u\rangle=2\pi \langle \mathcal{F}_\theta(u),u\rangle$;\\[-3ex]
\item[(b)] $\big(\mathrm{Re}(\langle S_n(\theta),u\rangle),\mathrm{Im}(\langle S_n(\theta),u\rangle)\big)/\sqrt{n}\stackrel{d}{\to} \mathcal{N}_2(0,\pi \langle \mathcal{F}_\theta(u),u\rangle\times I_2)$;\\[-3ex]
%\item[(c)] $g^u(\theta)=\lim_{n\to\infty}\mathrm{Var}(\langle Z_n(\theta),u\rangle)$.
\end{description}
\end{Proposition}
\begin{proof}
We first show that there exists for all $u\in H$ a set $N_u\subset [-\pi,\pi]$ with $\lambda(N_u)=0$, such that on $M_u=[-\pi,\pi]\backslash N_u$ {\bf (a)} and {\bf(b)} hold.

Let $u\in $ and let $\mathcal{G}_{k}(u)$ be the filtration of the process $(\langle X_t,u\rangle\colon t\in\mathbb{Z})$. From the results in Peligrad and Wu~\cite{wu:2010} we obtain that 
$\mathrm{Var}(\langle S_n(\theta),u\rangle)/n\to 2\pi f^u(\theta)$ for some function $f^u(\theta)$ which is finite on $M_u$. More precisely, slightly adapting the proofs of  Lemmas~4.1.\ and 4.2.\ in their article we obtain that the $L^2(\Omega)$ limit
$$
Z^u(\theta):=\lim_{n\to\infty}\frac{1}{\sqrt{2\pi}}\sum_{t=0}^n \mathcal{P}_0(\langle X_t,u\rangle )e^{-\ii t\theta}=\lim_{n\to\infty}\langle Z_n(\theta),u\rangle
$$
exists on $M_u$ and that
$
f^u(\theta)=\mathrm{Var}\big(Z^u(\theta)\big).
$
(Directly using their arguments would require to use the projection operator  $\mathcal{P}_0^u(\cdot)=E[\,\cdot\, |\mathcal{G}_0(u)]-E[\,\cdot\, |\mathcal{G}_{-1}(u)]$.)  We assume without loss of generality that $M=M_0\cap M_1$ is a subset of $M_u$, otherwise replace $M_u$ by $M_u\cap M$.
We now determine $f^u(\theta)$.
By result (ii) of Section~\ref{1} it follows that $E(\langle Z_{n_k}(\theta),u\rangle -\langle Z(\theta),u\rangle)^2\to 0$ for every $u\in H$ and all $\theta\in M_1$. Hence, by result (i) in the same section, we get
$$
%\mathrm{Var}\big(Z_{n_k}^u(\theta)\big)=
\mathrm{Var}\big(\langle Z_{n_k}(\theta),u\rangle\big)\to \langle \mathcal{F}_{\theta}(u),u\rangle<\infty,
$$
for all $\theta\in M$ and all $u\in H$, which implies on $M_u$ the relation
$
f^u(\theta)= \langle \mathcal{F}_{\theta}(u),u\rangle<\infty.
$
This shows part {\bf (a)} on $M_u$.

We have
$E[\langle X_t,u\rangle|\mathcal{G}_{-\infty}]=\langle E[X_t|\mathcal{G}_{-\infty}],u\rangle$ and by  Assumption~\ref{ass} this is equal to zero. The tower property of conditional expectations implies $E[\langle X_t,u\rangle|\mathcal{G}_{-\infty}(u)]=0$ and hence on $M_u$ {\bf (b)} directly follows from \cite{wu:2010}. Let us note that their CLT result is stated for real valued time series, but this requirement is not needed. Hence we can apply it for the time series $(\langle X_t,u\rangle\colon t\geq 1)$ which takes values in $\mathbb{C}$ when $u\in H$. 

It remains to prove that for all $u$ in $H$ we can find a common exceptional set of Lebesgue measure 0. To this end let $H^\prime$ be a dense and countable subset of $H$. We set $\widetilde{M}=\cap_{u^\prime\in H^\prime} M_u$. Then $[-\pi,\pi]\backslash \widetilde{M}$ has Lebesgue measure 0. Furthermore,  for all $u^\prime\in H^\prime$ and $\theta\in \widetilde{M}$ {\bf (a)} and {\bf (b)} hold. The objective is now to extend this result to all $u\in H$. For {\bf (a)} we observe that
\begin{align*}
&\big|\langle \mathcal{F}_{n;\theta}(u),u\rangle-\langle \mathcal{F}_{\theta}(u),u\rangle\big|\\
&\quad \leq\big|\langle \mathcal{F}_{n;\theta}(u),u\rangle-\langle \mathcal{F}_{n;\theta}(u^\prime),u^\prime\rangle\big|+ \big|\langle \mathcal{F}_{\theta}(u^\prime),u^\prime\rangle-\langle \mathcal{F}_{\theta}(u),u\rangle\big|\\
&\qquad+ \big|\langle \mathcal{F}_{n;\theta}(u^\prime),u^\prime\rangle-\langle \mathcal{F}_{\theta}(u^\prime),u^\prime\rangle\big|\\
&\quad\leq \Big[\mathrm{tr}\big(\mathcal{F}_{n;\theta}\big)+ \mathrm{tr}\big(\mathcal{F}_{\theta}\big)\Big]\times\Big[\big(\|u\|+\|u^\prime\|\big)\times\|u-u^\prime\|\Big]\\
&\qquad+ \big|\langle \mathcal{F}_{n;\theta}(u^\prime),u^\prime\rangle-\langle \mathcal{F}_{\theta}(u^\prime),u^\prime\rangle\big|.
\end{align*}
Since we can assume without loss of generality that $M_2\subset \widetilde{M}$, it follows that for all $\theta\in\widetilde{M}$
$$
\limsup_{n\to\infty}\big|\langle \mathcal{F}_{n;\theta}(u),u\rangle-\langle \mathcal{F}_{\theta}(u),u\rangle\big|\leq 4\varepsilon(\|u\|+1)\, \mathrm{tr}\big(\mathcal{F}_\theta\big),
$$
if $\|u-u^\prime\|\leq \varepsilon\leq 1$. Since $\varepsilon$ can be chosen arbitrarily small result {\bf (a)} follows.

The proof of part {\bf (b)} follows along similar lines of arguments. Just compare the characteristic functions of the real and complex part of $\langle S_n(\theta),u\rangle$ to the corresponding normal ones.
%$$
%|E\big\{\exp(\ii\, \mathrm{Re}(\langle S_n(\theta),u\rangle))-\exp(\ii\, N(0,\pi\langle \mathcal{F}_{\theta}(u),u\rangle)\big\}|.
%$$
%We take some $u^\prime\in H'$ close to $u$ and use that this term is bounded by the sum of
%\begin{align*}
%a_{1n}:=&|E\big\{\exp(\ii\, \mathrm{Re}(S_n^u(\theta)/\sqrt{n}))-\exp(\ii\, \mathrm{Re}(S_n^{u^\prime}(\theta)/\sqrt{n}))\big\}|;\\
%a_{2n}:=&|E\big\{\exp(\ii\, \mathrm{Re}(S_n^{u^\prime}(\theta)/\sqrt{n}))-\exp(\ii\, N(0,\pi\langle \mathcal{F}_{\theta}(u^\prime),u^\prime\rangle)\big\}|;\\
%a_{3n}:=&|E\big\{\exp(\ii\, N(0,\pi\langle \mathcal{F}_{\theta}(u^\prime),u^\prime\rangle)-\exp(\ii\, N(0,\pi\langle \mathcal{F}_{\theta}(u),u\rangle)\big\}|.
%\end{align*}
%Then $\limsup_n a_{2n}=0$. Furthermore, by $|1-e^{\ii x}|\leq |x|$ we obtain
%$$
%a_{1n}\leq E|\mathrm{Re}(\langle S_n(\theta)/\sqrt{n},u-u^\prime\rangle)|\leq\sqrt{\langle \mathcal{F}_{n;\theta}^X(u-u^\prime),u-u^\prime\rangle}.
%$$
%Moreover we have
%\begin{align*}
%\langle \mathcal{F}_{n;\theta}^X(u-u^\prime),u-u^\prime\rangle &\leq \|u-u^\prime\|^2\times \Vert \mathcal{F}_{n;\theta}^X\Vert_{\mathcal{N}}\\
%&=\|u-u^\prime\|^2\times \sum_{|h|<n}\left(1-\frac{|h|}{n}\right)E\langle X_h,X_0\rangle e^{-\ii h\theta}.
%\end{align*}
%According to our Lemma~\ref{le1b} this last sum converges for all $\theta\in M_4:=[-\pi,\pi]\backslash N_4$, where again $N_4$ has Lebesgue measure 0. Choosing $\varepsilon>0$ and $u^\prime$ close to $u$, we can find have for all $\theta\in \widetilde M:=M_3\cap M_4$, that $\limsup_n a_{1n}<\varepsilon$. The term $a_{3n}$ can be bounded in a similar way.
\end{proof}

\subsection{Tightness}\label{se:tight}

\begin{Lemma}\label{le6}
Consider sequences $(p_j^{(n)}\colon j\geq 1)$, $n\geq 0$, with the following properties: (a) $p_j^{(n)}\geq 0$ for all $j,n$; (b) $\lim_n p_j^{(n)}=p_j^{(0)}$; (c) $\sum_{j\geq 1} p_j^{(0)}=p<\infty$; (d) $\lim_n\sum_{j\geq 1} p_j^{(n)}=p$; (e) $\sum_{j\geq 1} p_j^{(n)}<\infty$ for all $n\geq 1$. Then
$$
\lim_{m\to\infty}\sup_n\sum_{j> m} p_j^{(n)}=0.
$$
\end{Lemma}
\begin{proof}
Fix an $\varepsilon>0$. We have to show that for $m\geq m(\varepsilon)$ we have $\sum_{j\geq m} p_j^{(n)}<\varepsilon$ for all $n\geq 1$.

By (c) we can choose $m_1=m_{1}(\varepsilon)$ such that $\sum_{j\geq m} p_j^{(0)}<\varepsilon/3$ for all $m\geq m_{1}$. 
Furthermore, by (b) we can choose a large enough $n_1=n_1(\varepsilon)$ such that for all $n\geq n_1$ we have $|\sum_{j=1}^{m_1}p_j^{(0)}-\sum_{j=1}^{m_1}p_j^{(n)}|<\varepsilon/3$. Next, by possibly further enlarging $n_1$ we deduce from (c) and (d) that $|\sum_{j\geq 1}p_j^{(n)}-\sum_{j\geq 1}p_j^{(0)}|<\varepsilon/3$.
Consequently, for $n\geq n_1$, we have
\begin{align*}
\sum_{j> m_1}p_j^{(n)}&=\sum_{j\geq 1}p_j^{(n)}-\sum_{j=1}^{m_1}p_j^{(n)}\\
&\leq |\sum_{j\geq 1}p_j^{(n)}-\sum_{j\geq 1}p_j^{(0)}|+|\sum_{j\geq 1}p_j^{(0)}-\sum_{j=1}^{m_1}p_j^{(0)}|+|\sum_{j=1}^{m_1}p_j^{(0)}-\sum_{j=1}^{m_1}p_j^{(n)}|\\
&<\varepsilon.
\end{align*}
Because of (a) this bound is still valid for all $m\geq m_1$.
For the $n_1$ just chosen, we can find an $m_2=m_2(\varepsilon)$, such that
$\sum_{j>m_2}p_j^{(n)}<\varepsilon$ for all $n\leq n_1$. This is because of (d) and (e) we know that $\sup_{n\geq 1}\sum_{j\geq 1}p_j^{(n)}<\infty$. And again, because of (a) we know also that $\sum_{j>m}p_j^{(n)}<\varepsilon$ for all $m\geq m_2$ and $n\leq n_1$. Hence, set $m(\varepsilon)=\max\{m_1,m_2\}$. 
\end{proof}

\begin{Lemma}\label{le4}
Take some ONB $(v_j)$ of $H$. Lemma~\ref{le6} applies with $p_j^{(n)}=\langle \mathcal{F}_{n;\theta}(v_j),v_j\rangle$, $p_j^{(0)}=\langle \mathcal{F}_{\theta}(v_j),v_j\rangle$ for all $\theta\in\widetilde M$.
\end{Lemma}
\begin{proof}
We can assume that $v_j\in H^\prime$ for all $j\geq 1$.
Relation (a) is trivial. Relation (b) follows from part {\bf (a)} of Proposition~\ref{prop1}. Relation (c) holds because $\mathcal{F}_{\theta}$ is nuclear on $\widetilde M$. And similarly relation (e) holds because apparently $\mathcal{F}_{n;\theta}$ is nuclear for any $n$. Finally note that (d) can be reformulated as 
$
\mathrm{tr}\big(\mathcal{F}_{n;\theta}\big)\to \mathrm{tr}\big(\mathcal{F}_{\theta}\big).
$
By Lemma~\ref{le1b} this holds for almost all $\theta\in M_2\subset\widetilde{M}$.
%By construction of $\mathcal{F}_{\theta}$ and by the proof of Lemma~\ref{le1a} we have
%\begin{align*}
%E\|X_0\|^2&=\lim_{n\to\infty}E\|X_0-E[X_0|\mathcal{F}_{-n}]\|^2\\
%&=\sum_{t=1}^n E\|\mathcal{P}_{-t}(X_0)\|^2\\
%&=\lim_{n\to\infty}\vertiii{Z_n}^2=\vertiii{Z}^2=\int_{-\pi}^\pi \mathrm{tr}(\mathcal{F}_{\theta})d\theta.
%\end{align*}
%Using Fatou's lemma and the fact that $\langle v_j,\mathcal{F}_{n;\theta}^X(v_j)\rangle\to \langle v_j, \mathcal{F}_{\theta}(v_j)\rangle$ we get
%\begin{align*}
%&\mathrm{tr}(\mathcal{F}_{\theta})\leq\liminf_n\sum_{j\geq 1}\langle v_j,\mathcal{F}_{n;\theta}^X(v_j)\rangle=\liminf_n\mathrm{tr}(\mathcal{F}_{n;\theta}^X)\\
%&\quad=\liminf_n \sum_{|h|<n}\left(1-\frac{|h|}{n}\right)E\langle X_h,X_0\rangle e^{-\ii h\theta}:=b(\theta).
%\end{align*}
%Again by Fatou's lemma it follows that
%$$
%\int_{-\pi}^\pi b(\theta)d\theta=\int_{-\pi}^\pi \liminf_n \sum_{|h|<n}\left(1-\frac{|h|}{n}\right)E\langle X_h,X_0\rangle e^{-\ii h\theta}d\theta\leq E\|X_0\|^2.
%$$
%Assumption~\ref{ass1} implies that $b(\theta)=\lim_n \mathrm{tr}(\mathcal{F}_{n;\theta}^X)$. Since $\mathrm{tr}(\mathcal{F}_{\theta})$ and $b(\theta)$ are non-negative, this implies that $\lim_n \mathrm{tr}(\mathcal{F}_{n;\theta}^X)=\mathrm{tr}(\mathcal{F}_{\theta})$ for almost all $\theta$. Without loss of generality this set of $\theta$ is contained in $\widetilde M$.
\end{proof}

\begin{Lemma}\label{prop2}
Under Assumptions~\ref{ass} the sequence $(S_n(\theta)/\sqrt{n}\colon n\geq 1)$ is tight for all $\theta\in\widetilde{M}$.
\end{Lemma}
\begin{proof}
Let $\varepsilon >0$. We consider the sequences $0< l_k \nearrow \infty$ and  $0< N_k \nearrow \infty$ and define
$$
K = \bigcap_{k=1}^{\infty} \left\lbrace  x\in H: \sum_{j> N_k} \vert\langle v_j,x \rangle\vert^2 \leq \frac{1}{l_k}  \right\rbrace.
$$
Just as in \cite{bosq:2000} (p.\ 52) we can see that it is a compact subset of $H$. We now have that
\begin{align*}
P \left(S_n(\theta)/\sqrt{n} \in K \right) &\geq 1- \sum_{k=1}^{\infty} l_k 
\sum_{j>N_k} E\big\vert\langle S_{n}(\theta)/\sqrt{n},v_j \rangle\big\vert^2 \\
&= 1- \sum_{k=1}^{\infty} 2\pi l_k \sum_{j>N_k}  \langle  \mathcal{F}_{n;\theta}(v_j),v_j \rangle,
\end{align*}
where we used the $\sigma$-subadditivity and the Markov inequality. By Lemma \ref{le4} we know that
$$
\sup_n\sum_{j\geq m}\langle  \mathcal{F}_{n;\theta}(v_j),v_j \rangle\to 0\quad(m\to\infty).
$$
Therefore, for any $\varepsilon>0$, we can choose increasing sequences $(l_k)$ and $(N_k)$ such that $$  2\pi l_k \sum_{j>N_k}  \langle  \mathcal{F}_{n;\theta}(v_j),v_j \rangle\leq \varepsilon 2^{-k}.$$
\end{proof}

\subsection{Proofs of Theorems~\ref{main} and \ref{main2}}

\begin{proof}[Proof of Theorem~\ref{main}]
Parts {\bf (II)} and {\bf (III)} of Theorem~\ref{main} follow directly from Lemmas~\ref{le1b} and \ref{le5}. Part {\bf (I)} can be deduced from the polarization identity for self-adjoint operators $\Gamma$
\begin{align*}
\langle \Gamma (x),y \rangle &= \frac{1}{4}\big[ \langle \Gamma (x+y),x+y \rangle -\langle \Gamma (x-y),x-y \rangle\\
&\qquad +\ii \langle \Gamma (x+\ii y),x+\ii y \rangle -\ii \langle \Gamma (x-\ii y),x-\ii y \rangle \big],
\end{align*}
and part {\bf (a)} of Proposition~\ref{prop1}. Next, the asymptotic normality of $S_n(\theta)\sqrt{n}$ for all $\theta\in\widetilde{M}$ follows from the corresponding convergence of the projections (Proposition~\ref{prop1}, part {\bf (b)}) and the tightness shown in Lemma~\ref{prop2}. Finally, the asymptotic independence relation {\bf (IV)} can be obtained by verifying it for the projections $\langle S_n(\theta),u\rangle/\sqrt{n}$ and $\langle S_n(\theta^\prime),u\rangle/\sqrt{n}$. For this we can refer to~\cite{wu:2010}.
\end{proof}

\begin{proof}[Proof of Theorem~\ref{main2}]
Let $\theta \in [-\pi,\pi]$ be such that Assumption~\ref{ass2} is satisfied. Due to relation~\eqref{repSn} we have that
\begin{align*}\frac{E\Vert S_n(\theta)\Vert^2}{n} &=\frac{2\pi}{n}\,\sum_{k=1}^n E\Vert Z^{(0)}_{n-k}(\theta)\Vert^2 + \frac{1}{n}\, E\big\|E[S_n(\theta)\vert \mathcal{G}_0]\big\|^2 \\
&= 2\pi \, E\Vert Z^{(0)}(\theta)\Vert^2 \; + \; o(1),
\end{align*}
where we used both \textbf{(A1)} and \textbf{(A2)}. Since $\mathrm{tr}(\mathcal{F}_{\theta})=E\Vert Z^{(0)}(\theta)\Vert^2$ we conclude that {\bf (II)} holds for the fixed $\theta$. Note that condition~\textbf{(A1)} just provides a stronger version of Lemma~\ref{le1}. There we only get the limit of $Z_n$ ($=Z_n^{(0)}$) in $L_G^2([-\pi,\pi])$, whereas now we get it pointwise by Assumption {\bf(A1)}. In other words, we got rid of the exceptional sets $N_1$ and $N_2$ in Lemmas~\ref{le1} and~\ref{le1b}. The proof of Lemma~\ref{le5} is unchanged, hence {\bf (III)} follows. \medskip

Proposition~\ref{prop1} can be proven even more easily for a given $\theta \in [-\pi,\pi]$. For part \textbf{(a)} we project \eqref{repSn} onto $u\in H$ and deduce just as above that
\[\frac{E\vert \langle S_n(\theta),u\rangle \vert^2}{n} = 2\pi \, E\vert \langle Z^{(0)}(\theta),u\rangle\vert^2 \; + \; o(1).\]
This also shows {\bf (I)} by the polarization identity.  Part \textbf{(b)} can be shown by the same martingale approximation as in the proof of Theorem~2.1. of~\cite{wu:2010}. Alternatively, one may directly apply Theorem~2 in~\cite{wu:2004}.

\medskip

Finally the tightness of $S_n(\theta)/\sqrt{n}$ can be shown using the same proof as in Lemma~\ref{prop2} since it is now clear that relations (a)-(e) of Lemma~\ref{le6} are satisfied for the particular $\theta$.
\end{proof}

\section*{Acknowledgement}
This research was supported by the
Communaut\'e fran\c{c}aise de Belgique---Actions de Recherche
Concert\'ees (2010--2015) and from the Interuniversity Attraction Poles Programme (IAP-network P7/06), Belgian Science Policy Office. Cl\'ement Cerovecki acknowledges support from the F.R.S.-FNRS Fond de la Recherche Scientifique, Rue d'Egmont 5, B-1000 Bruxelles.

\bibliographystyle{plainnat}

\begin{thebibliography}{10}

\bibitem[Aue et al.(2015)]{aue:dubart:hormann:2015} {\sc Aue, A., Dubart Norinho, D., and H\"ormann, S.} (2015). {\it On the prediction of stationary functional time series.}
Journal of the American Statistical Association, 110, 509, 378--392.


\bibitem[Berkes et al.(2013)]{berkes:horvath:rice:2013} {\sc Berkes, I., Horv\'ath, L. and Rice, G.} (2013). {\it Weak invariance principles for sums of dependent
random functions.}  Stochastic Processes and their Applications, 123, 385--403. 

\bibitem[Bosq(2000)]{bosq:2000} {\sc Bosq, D.} (2000). {\it Linear Processes in Function Spaces}.
Springer, New York.

\bibitem[Brockwell and Davis(1991)]{brockwell:davis:1991} {\sc Brockwell, P.J. and Davis, R.A.} (1991). Time Series: Theory and Methods. Springer, New York.

\bibitem[Bradley(2005)]{bradley:2005} {\sc Bradley, R.C.} (2005). {\it Basic Properties of Strong Mixing
Conditions. A Survey and Some Open Questions}. Probability Surveys Vol. 2, 107--144.

\bibitem[Chatterji(1964)]{chatterji:1964} {\sc Chatterji, S.D.} (1964). {\it A note on the convergence of Banach-space valued martingales}.
Mathematische Annalen, Volume 153, Issue 2, pp 142--149.

\bibitem[H\"ormann et al.(2015a)]{hoermann:kidzinski:hallin:2014} {\sc H\"ormann, S., Kidzi\'nski, L. and Hallin, M.} (2015). {\it Dynamic Functional Principal Component.} Journal of the Royal Statistical Society: Series B, 77, 319--348.

\bibitem[H\"ormann et al.(2015b)]{hoermann:kidzinski:kokoszka:2014} {\sc H\"ormann, S., Kidzi\'nski, L. and Kokoszka, P.} 
 {\it Estimation in functional lagged regression.} Journal of Time Series Analysis. (Forthcoming. DOI: 10.1111/jtsa.12114.)

\bibitem[H\"ormann and Kokoszka(2010)]{hormann:kokoszka:2010} {\sc H\"ormann, S., Kokoszka, P.} (2010). {\it Weakly dependent functional data.} The Annals of Statistics, 38, 1845--1884.

\bibitem[Horv\'ath et al.(2014)]{horvath:kokoszka:rice:2014} {\sc Horv\'ath, L., Kokoszka, P., and Rice, G.} (2014). {\it Testing stationarity of functional time series.} Journal of Econometrics, 179, 66--82.


\bibitem[Horv\'ath et al.(2015)]{horvath:rice:whipple:2015} {\sc Horv\'ath, L., Rice, G., and Whipple, S.} (2014). {\it Adaptive bandwidth selection in the estimation of the long run covariance of functional time series.} Computational Statistics and Data Analysis. Forthcoming.

\bibitem[Hyndman and Shang(2009)]{hyndman:shang:2009} {\sc Hyndman, R.J. and Shang, H.L.} (2009).  {\it Forecasting functional time series.}
Journal of the Korean Statistical Society, 38, No. 3, 199--211.


\bibitem[Jirak(2013)]{jirak:2013} {\sc Jirak, M.} (2013). {\it On weak invariance principles for sums of dependent random functionals.} Statistics and Probability Letters, 83, 2291--2296.

\bibitem[Merlev\`ede et al.(1997)]{merlevede:peligrad:utev:1997} {\sc Merlev\`ede, F., Peligrad, M. and Utev, S.} (1997).
{\it  Sharp Conditions for the CLT of Linear Processes in
a Hilbert Space.} Journal of Theoretical Probability, 10, No. 3, 681--693.

\bibitem[Mikusi\'nski(1978)]{mikusinski} {\sc Mikusi\'nski, J.} (1978). {\it The Bochner Integral}. Birkhauser, Basel.

\bibitem[Panaretos and Tavakoli(2013a)]{panaretos:tavakoli:2013a} {\sc Panaretos, V.M. and Tavakoli, S.} (2013). {\it Fourier analysis of stationonary time series in function spaces.} The Annals of Statistics, 41, No.~2, 568--603.

\bibitem[Panaretos and Tavakoli(2013b)]{panaretos:tavakoli:2013b}
{\sc Panaretos, V.M. and Tavakoli, S.} (2013). {\it Cram{\'e}r--{K}arhunen--{L}o{\`e}ve representation and harmonic
  principal component analysis of functional time series.} Stochastic Processes and their Applications, 123, 2779--2807.
  
\bibitem[Peligrad and Wu(2010)]{wu:2010} {\sc Peligrad, M. and Wu, W.B.} (2010). {\it Central limit theorem for Fourier transforms of stationary processes.} The Annals of Probability, 38, No. 5, 2009--2022.

\bibitem[P\"otscher and Prucha(1997)]{poetscher:prucha:1997} {\sc P\"otscher, B.M. and Prucha, I.R.} (1997). {\it Dynamic Nonlinear Econometric Models.} Springer.

\bibitem[Ra\v{c}kauskas and Suquet(2010)]{rackauskas:suquet:2010} {\sc Ra\v{c}kauskas, A. and Suquet, Ch.} (2010). {\it On limit theorems for Banach-space-valued linear processes.} Lithuanian Mathematical Journal, 50, No.~1, 71--87.

\bibitem[Walker(1965)]{walker:1965} {\sc Walker, A.M.} (1965). {\it Some asymptotic results for the periodogram of a stationary time series.} Journal of the Australian Mathematical Society, 5, No.\ 1, 107--128.

\bibitem[Wu(2004)]{wu:2004} {\sc Wu, W.B.} (2004). {\it Fourier transforms of stationary processes.} Proceeding of the American society, Vol 133, Number 1, Pages 285--293.


\end{thebibliography}

\end{document}